\documentclass[amssymb]{article}
\usepackage[english]{babel}
\usepackage{amsmath}
\usepackage{amssymb}
\usepackage{amsfonts}
\usepackage{amsthm}
\usepackage[T2A]{fontenc}
\usepackage[matrix,arrow,curve]{xy}
\usepackage{mathrsfs}
\usepackage{bm}

\usepackage{hyperref}
\begin{document}

\newtheorem{thm}{Theorem}
\newtheorem{lm}{Lemma}
\newtheorem{cor}{Corollary}
\newtheorem{prop}{Proposition}
\newtheorem*{notation}{Notation}
\newtheorem{rem}[thm]{Remark}
\newtheorem{Claim}[thm]{Claim}
\newtheorem{fact}{Fact}
\theoremstyle{definition}
\newtheorem{df}{Definition}
\newtheorem{exam}{Example}

\author{N.~L.~Polyakov, HSE University}
\title{A note on the Canonical Ramsey Theorem and  Ramsey ultrafilters}
\date{}
\maketitle
\begin{abstract}  We give a characterizations of Ramsey ultrafilters on $\mathscr P(\omega)$ in terms of functions $f:\omega^n\to\omega$ and their ultrafilter extensions. To do this, we prove that for any partition $\mathcal P$ of $[\omega]^n$ there is a finite partition~$\mathcal Q$ of $[\omega]^{2n}$ such that any set $X\subseteq \omega$ that is homogeneous for $\mathcal Q$ is a finite union of sets that are canonical for $\mathcal P$.
\end{abstract}

Ramsey theory studies homogeneous (monochromatic) substructures for various partitions (colorings) of structures. In  Canonical Ramsey Theory, a more general notion of a canonical substructure is considered. A nonprincipal ultrafilter on $\mathscr P(\omega)$ that contains an homogeneous set for any finite partition $\mathcal P$ of $[\omega]^n$, $1\leq n<\omega$, is called a Ramsey ultrafilter. It is well known that an ultrafilter $\mathfrak u$ on $\mathscr P(\omega)$ is a Ramsey ultrafilter if and only if it is selective and if and only if it is minimal (with respect to the  Rudin-Keisler (pre)order), see, for example, \cite{Comfort Negrepontis}. Both of these characterizations are formulated in terms of functions $f:\omega\to\omega$ or their ultrafilter extensions. We show that a nonprincipal ultrafilter on $\mathscr P(\omega)$ is a Ramsey ultrafilter if and only if it contains a canonical set for each  partition of $[\omega]^n$, and we give characterizations of Ramsey ultrafilters  similar to the above in terms of functions on $\omega$ of arbitrary finite arity and the concept of their ultrafilter extensions, which appeared in recent works \cite{Goranko, Saveliev}. To do this, we prove that for any partition $\mathcal P$ of $[\omega]^n$, $1\leq n<\omega$, there is a finite partition~$\mathcal Q$ of $[\omega]^{2n}$ such that any set $X\subseteq \omega$ that is homogeneous for $\mathcal Q$ is a finite union of sets that are canonical for $\mathcal P$. This theorem provides an easy way to go from some results of Ramsey Theory to similar results of Canonical Ramsey Theory.

\subsection*{1. Combinatorial Theorem}

\par Everywhere below, we identify natural numbers and finite ordinals.  The set of all finite ordinals is denoted by $\omega$. We use the fact that each ordinal is a set of ordinals that precede~it. For example, for any $X\subseteq \omega$ and $x\in \omega$ the term $x\cap X$ denotes the set $\{y\in X: y<x\}$.

\par For any set $X$ and $n\in\omega$ the set of all $n$-element subsets of $X$ is denoted by~$[X]^n$:
$$
[X]^n=\{\bm x\subseteq X: |\bm x|=n\}.
$$
A set $\mathcal P\subseteq \mathscr P(Z)$ is called a \textit{partition} of a set $Z$ if 
\begin{enumerate}
\item $\bigcup \mathcal P=Z$, and
\item $(\forall X, Y\in \mathcal P)\, X\cap Y= \emptyset \vee X=Y$.
\end{enumerate}
For convenience, we assume that one of the elements of the partition can be empty\footnote{Many results in Ramsey theory are formulated in terms of colorings of sets. Partition terminology and coloring terminology can be used interchangeably. A \textit{coloring} of a set $Z$ is any function $f: Z\to C$ for some set $C$ (of colors). Each coloring of $Z$ defines a partition $\mathcal P_f=\{f^{-1}(c): c\in C\}$, and each partition $\mathcal P$ of $Z$ defines a coloring $f_\mathcal P: Z\to \mathcal P$ with $z\in f(z)$. The maps $\mathcal P\mapsto f_\mathcal P$ and $f\mapsto \mathcal P_f$ are mutually inverse.}. 
\begin{df}
For any partition $\mathcal P$ of $[X]^n$, a set $Y\subseteq X$ is called \emph{homogeneous} for $\mathcal P$ if there is a set $P\in \mathcal P$ such that $[Y]^n\subseteq P$. 
\end{df}
For $n\in \omega$,  let $\mathrm{RT}\,(n)$ assert that for any finite partition $\mathcal P$ of $[\omega]^n$ there exists an infinite set $Y\subseteq \omega$ that is homogeneous for $\mathcal P$. The (infinite) \textit{Ramsey Theorem} (RT) \cite{Ramsey}  states that $\mathrm{RT}\,(n)$ holds for every $n$, $1\leq n<\omega$. The Ramsey Theorem gave rise to an extensive branch of mathematics, which is called \textit{Ramsey Theory}, see, eg.~\cite{Graham2015}. Many different versions and extensions of the Ramsey Theorem can be found in  publications. 
\par The Canonical Ramsey Theorem of Erd\H{o}s and Rado \cite{ErdosRado1950} is a natural generalization of  the Ramsey Theorem to the case of arbitrary partitions of $[\omega]^n$.  
\par  For any partition $\mathcal P$ of a set $Z$, the corresponding equivalence relation is denoted by~$\approx_\mathcal P$:
$$
x\approx_\mathcal P y\,\Leftrightarrow\, (\exists P\in \mathcal P)\, x, y\in P
$$
for all $x,y\in Z$.

\par For any $X\subseteq \omega$ and $i<|X|$, the $i$-th (in the natural ordering) element $x\in X$ is denoted by $X_{[i]}$:
$$
x=X_{[i]}\,\Leftrightarrow\, (x\in X\wedge |x\cap X|=i).
$$

\begin{df} Let $\mathcal P$ be a partition of $[\omega]^n$, $1\leq n<\omega$, and $I\subseteq n$. A set $X\subseteq \omega$ is called $I$-\emph{canonical} for $\mathcal P$   if 
$$
\bm p \approx_\mathcal P \bm q\,\Leftrightarrow\, \bigwedge\limits_{i\in
I}(\bm p_{[i]}=\bm q_{[i]})
$$
for all $\bm p, \bm q\in [X]^n$. A set $X\subseteq \omega$ is called \textit{canonical} for $\mathcal P$  if there 
is  a set $I\subseteq n$ such that~$X$ is $I$-\emph{canonical} for $\mathcal P$.
\end{df}

\par For $n\in \omega$,  let $\mathrm{CRT}\,(n)$ assert that for any partition $\mathcal P$ of $[\omega]^n$ there exists an infinite set $Y\subseteq \omega$ that is canonical for $\mathcal P$.  The \textit{Canonical Ramsey Theore}m (CRT) affirms that $\mathrm{CRT}\,(n)$ holds for every $n$, $1\leq n<\omega$. 

\par As usual, the empty conjunction is true, so any
$\emptyset$-canonical set $X$ is homogeneous (for any partition $\mathcal P$
of $[\omega]^n$).  Any infinite canonical set $X\subseteq \omega$ for a finite partition $\mathcal P$ of $[\omega]^n$ is $\emptyset$-canonical for $\mathcal P$. Therefore, RT immediately follows from CRT. 

\par Several proofs  of CRT appeared in print: the original proof of Erd\H{o}s and Rado \cite{ErdosRado1950}, the simpler version of Rado \cite{Rado1986}, the proof of Mileti \cite{Mileti2008}, etc. (see also \cite{ErdosRado1952,Lefmann1995} for the finite version). The recent work of Matet \cite{Matet2016} contains an elegant proof using the antilexicographic ordering on $[\omega]^n$. 

\par The proofs in  \cite{ErdosRado1950} and \cite{Rado1986} use the following strategy. To prove $\mathrm{CRT}\,(n)$, given a partition $\mathcal P$ of  $[\omega]^n$, the authors construct a special finite partition $\mathcal Q$ of $[\omega]^{2n}$, and use $\mathrm{RT}\,(2n)$. However, as far as we know, a relevant correspondence between partitions of $[\omega]^n$ and finite partitions of $[\omega]^{2n}$ is nowhere stated explicitly. We prove that the following statement holds. 
\begin{thm}\label{main_var} For any natural number $n\geq 1$ and partition $\mathcal P$ of $[\omega]^n$ there is a finite partition~$\mathcal Q$ of $[\omega]^{2n}$ such that any set $X\subseteq \omega$ that is homogeneous for $\mathcal Q$ is a finite union of sets that are canonical for $\mathcal P$.
\end{thm}
\par CRT immediately follows from RT and Theorem \ref{main_var}. Thus, Theorem \ref{main_var} provides one more proof of CRT. The proof of Theorem \ref{main_var} is quite elementary and does not use RT. Therefore, informally speaking, we divide CRT into a Ramseyan and a non-Ramseyan parts. This approach is particularly useful in the theory of ultrafilters, see Section 2. 

\par We obtain Theorem \ref{main_var} as a formal logical consequence of the following combinatorial fact, which provides more information and can be of independent value.  
\par Let $\mathcal P$ be a partition of $[\omega]^n$, $n\geq 1$. For any $\bm p, \bm q\in [2n]^n$, let
$$
Q_{\bm p \bm q}= \left\{\bm z\in [\omega]^{2n}: \left\{\bm
z_{[i]}:i\in \bm p\right\}\approx_\mathcal P \left\{\bm z_{[i]}:i\in
\bm q\right\}\right\}.
$$
Let $\mathcal P^\ast$ be the set of all atoms of the (finite) field of
sets $\mathcal A$ generated by all sets $Q_{\bm p \bm q} $, \mbox {$
\bm p, \bm q \in [2n]^n $}, i.e., $Q\in \mathcal P^\ast$ iff $Q$ is
a non-empty subset of $[\omega]^{2n}$ representable in the form
$$
\bigcap\limits_{\bm p, \bm q \in [2n]^n}\!\!\!\!S_{\bm p \bm q}
$$
where $S_{\bm p \bm q}=Q_{\bm p \bm q}$ or
$S_{\bm p \bm q}=[\omega]^{2n}\setminus Q_{\bm p \bm q}$ for any $\bm p, \bm q \in
[2n]^n$. 
\par Obviously, $\mathcal P^\ast$ is a finite partition of $[\omega]^{2n}$ of cardinality at most
$2^{\frac{1}{2}\binom{2n}{n}\left(\binom{2n}{n}-1\right)}$ where $\binom{2n}{n}$ is a binomial coefficient, $\binom{2n}{n}=\dfrac{(2n)!}{(n!)^2}$.

\begin{thm}\label{main} For any natural number $n\geq 1$ there is a natural number~$m$ (we can put $m=n^{n\binom{2n}{n}\left(\binom{2n}{n}-1\right)}$) such that for any partition $\mathcal P$ of $[\omega]^n$ and any set $Q\in \mathcal P^\ast$ there is a set $I\subseteq n$ such that for any infinite set $X\subseteq \omega$ with $[X]^{2n}\subseteq Q$ there is a partition $\mathcal R=\{R_0, R_1, \ldots, R_m\}$ of~$X$ such that
\begin{enumerate}
\item the set $R_0$ is finite of cardinality at most $m$,
\item for any $i$, $1\leq i\leq m$, the set $R_i$ is an infinite $I$-canonical set for $\mathcal P$.
\end{enumerate}
\end{thm}

\begin{proof}

Let $n=1$. Put $m=1$. For any partition $\mathcal P$ of $[\omega]^1$ the partition $\mathcal P^\ast$ contains at most two sets:
$$
Q_0=\left\{\bm x\in [\omega]^2: \left\{\bm
x_{[0]}\right\}\approx_\mathcal P \left\{\bm
x_{[1]}\right\}\right\}\text{ and }Q_1=\left\{\bm x\in [\omega]^2: \left\{\bm
x_{[0]}\right\}\not\approx_\mathcal P \left\{\bm
x_{[1]}\right\}\right\}.
$$
Let $I_0=\emptyset$ and $I_1=\{0\}$. If $[X]^{2n}\subseteq Q_i$ then $X$ is $I_i$-canonical for $\mathcal P$, $i\in \{0, 1\}$. It remains to put $R_0=\emptyset$ and $R_1=X$.

\par Further we assume $ n \geq 2 $. Put $m=n^{n\binom{2n}{n}\left(\binom{2n}{n}-1\right)}$. Fix an arbitrary set $Q \in \mathcal P^\ast$. For any $\bm p, \bm q\in [2n]^n$ denote
$ I_{\bm p \bm q}=\{i<n: \bm p_{[i]}=\bm q_{[i]}\}$ and $
I(Q)=\bigcap\limits_{Q\subseteq Q_{\bm p \bm q}}I_{\bm p \bm q}.$ We shall prove that $I=I(Q)$ is required.

\par Fix a set $X \subseteq \omega$ such that $[X]^{2n}\subseteq Q$.  Denote $\min(X)=e$,  $X^-=X\setminus \{e\}$, and  $\bm x^+=\bm x\cup \{e\}$ for any set $\bm x\subseteq X$. For any $x,y\in X$ denote $\rho(x,y)=|(x\triangle y)\cap X|$. The function $\rho$ is a metric on $X$, and for all $x, y, z\in X$
$$
x\leq y\leq z\,\Rightarrow\, \rho(x, z)=\rho(x, y)+\rho(y, z).
$$
For any non-empty set $\bm x\subseteq X^-$, the natural number $$d(\bm x)=\min\left\{\rho(x,y): x,y\in \bm x^+, x\neq y\right\}$$ is called the \emph{sparsity} of $\bm x$. For definiteness, we can put $d(\emptyset)=\omega$.

\par Our immediate goal will be to prove the following Lemma.

\begin{lm}\label{mainlemma}
Let $\bm x, \bm y\in [X^-]^n$, and $\min(d\left(\bm x\right),d\left(\bm
y\right))\geq m$. Then $\bm x\approx_\mathcal P\bm y$ if and only if $\bm x_{[i]}=\bm
y_{[i]}$ for all $i\in I(Q)$.
\end{lm}

\begin{proof}

% First we introduce a number of technical concepts and prove some of their properties.
For each $\bm p, \bm q\in [2n]^n$ and $\bm x, \bm y\in [X]^n$ we write $\bm x\underset{\bm p \bm q}{\longrightarrow}\bm y$ if there is a set $\bm z\in [X]^{2n}$ such that $\{\bm z_{[i]}: i\in \bm p\}=\bm x$ and $\{\bm z_{[i]}: i\in \bm q\}=\bm y$. We write $\bm x\underset{Q}{\longleftrightarrow}\bm y$ if $\bm x\underset{\bm p \bm q}{\longrightarrow}\bm y$ for some $\bm p, \bm q\in [2n]^n$ for which $Q\subseteq Q_{\bm p\bm q}$. 
% Note that the relation $\underset{Q}{\longrightarrow}$ is symmetric and reflexive because $Q_{\bm p\bm q}=Q_{\bm q\bm p}$ ans $Q_{\bm p\bm p}=[\omega]^{2n}$ for all $\bm p, \bm q\in [2n]^n$. Furthermore, the following fact holds. 

\begin{fact}\label{Two_relations}
For all $\bm x, \bm y\in [X]^n$, $\bm x\underset{Q}{\longleftrightarrow}\bm y$ if and only if $\bm x\approx_\mathcal P \bm y
$.
\end{fact} 

\begin{proof} Let $\bm x, \bm y\in [X]^n$. If $\bm x\underset{Q}{\longleftrightarrow}\bm y$ there are sets $\bm p, \bm q\in [2n]^n$ and $\bm z\in [X]^{2n}$ such that $Q\subseteq Q_{\bm p\bm q}$, $\{\bm z_{[i]}: i\in \bm p\}=\bm x$ and $\{\bm z_{[i]}: i\in \bm q\}=\bm y$. Since $[X]^{2n}\subseteq Q$, we have:
$$
\bm x = \{\bm z_{[i]}: i\in \bm p\}\approx_\mathcal P \{\bm z_{[i]}: i\in \bm q\}=\bm y.
$$
Assume $\bm x\approx_\mathcal P \bm y$. Choose a set $\bm z\in [X]^{2n}$ such that $|\bm z|=2n$ and $\bm x\cup\bm y\subseteq \bm z$. Let $\bm p$ and $\bm q$ be the sets of numbers of elements of the sets $\bm x$ and $\bm y$, respectively, in the set $\bm z$:
$$
\bm p=\{|x\cap \bm z|: x\in \bm x\}\text{ and }\bm q=\{|y\cap \bm z|: y\in \bm y\}.
$$
In other words, $\{\bm z_{[i]}: i\in \bm p\}=\bm x$ and $\{\bm z_{[i]}: i\in \bm q\}=\bm y$. Suppose $Q\nsubseteq Q_{\bm p\bm q}$. Then, by the construction, $Q\subseteq [\omega]^{2n}\setminus Q_{\bm p\bm q}= \left\{\bm z\in [\omega]^{2n}: \left\{\bm
z_{[i]}:i\in \bm p\right\}\not\approx_\mathcal P \left\{\bm z_{[i]}:i\in
\bm q\right\}\right\}$. Since $[X]^{2n}\subseteq Q$, we have $\bm x\not\approx_\mathcal P\bm y$, a contradiction. Therefore,  $Q\subseteq Q_{\bm p\bm q}$.
\end{proof}
 
%We denote  the transitive closure of $\underset{Q}{\longrightarrow}$ by $\underset{Q}{\Longrightarrow}$, and  the reflexive, transitive and symmetric closure of  $\underset{Q}{\longrightarrow}$ by $\underset{Q}{\sim}$.
%\par Since $[X]^{2n}\subseteq Q$, we have that $\bm x\underset{Q}{\longrightarrow}\bm y$ implies  $\bm x\approx_\mathcal P \bm y$ for all $\bm x, \bm y\in [X]^n$. Therefore
%$$
%\bm x\underset{Q}{\sim}\bm y\,\Rightarrow\, \bm x\approx_\mathcal P \bm y.
%$$
Thus, it suffices to prove that  for
all $\bm x, \bm y\in [X^-]^n$ satisfying the conditions of
Lemma~\ref{mainlemma}, $\bm x\underset{Q}{\longleftrightarrow}\bm y$ if and only if $\bm x_{[i]}=\bm
y_{[i]}$ for all $i\in I(Q)$. In addition, we will use the fact that $\underset{Q}{\longleftrightarrow}$ is an equivalence relation. 

% \par For any $Y\subseteq \omega$ and $y\in Y$ the symbol
% $N_Y(y)$ denotes the sequence number of~$y$ in~$Y$ (w.r.t. the
% natural numbering): $ N_Y(y)=|y\cap Y| $. Thus, $Y_{[N_Y(y)]} = y$ for any $y\in Y$, and $N_Y(Y_{[i]})=i$ for any $i<|Y|$.

\begin{fact}\label{I_pq_stabl} Let $\bm x, \bm y\in [X]^n$, $\bm p, \bm q\in [2n]^n$, $\bm x\underset{\bm p\bm q}{\longrightarrow}\bm y$, and $i<n$. Then 
$$
i\in I_{\bm p\bm q}\,\Leftrightarrow\, \bm x_{[i]}=\bm y_{[i]}.
$$
\end{fact}
\begin{proof} By the condition, there is a set $\bm z\in [X]^{2n}$ such that $\bm x=\{\bm z_{\bm p_{[0]}}, \bm z_{\bm p_{[1]}}, \ldots, \bm z_{\bm p_{[n-1]}}\}$ and $\bm y=\{\bm z_{\bm q_{[0]}}, \bm z_{\bm q_{[1]}}, \ldots, \bm z_{\bm q_{[n-1]}}\}$. So, for any $i<n$,
$$
\bm x_{[i]}=\bm z_{[\bm p_{[i]}]}\text{ and }\bm y_{[i]}=\bm z_{[\bm q_{[i]}]}.
$$
Therefore,
$$
\bm x_{[i]}=\bm y_{[i]}\,\Leftrightarrow\, \bm z_{[\bm p_{[i]}]}=\bm z_{[\bm q_{[i]}]}\,\Leftrightarrow\, \bm p_{[i]}= \bm q_{[i]}\,\Leftrightarrow\, i\in I_{\bm p\bm q}.
$$
\end{proof}
From Facts \ref{Two_relations} and \ref{I_pq_stabl}, we have: for all $\bm x, \bm y\in [X]^n$, if  $\bm x\approx_\mathcal P\bm y$ then $\bm x_{[i]}=\bm y_{[i]}$ for all $i\in I(Q)$. Now we have to prove the reverse implication under the assumption $\bm x, \bm y\in [X^-]^n$ and $\min(d\left(\bm x\right), d\left(\bm
y\right))\geq m$. We start by proving that for any sufficiently sparse set $\bm x\in [X^-]^n$ there is a set $\bm y\in [X^-]^n$ such that  $\bm x\underset{Q}{\longleftrightarrow}\bm y$ and $\bm x\cap \bm y$ is exactly $\{\bm x_{[i]}: i\in I(Q)\}$.

\par For any finite sets $\bm x, \bm y\subseteq X^-$, let
$$
r(\bm x, \bm y)=\left\{
\begin{array}{cl}
0&\text{if $\bm y\subseteq \bm x$}\\
\max\limits_{y\in \bm y\setminus \bm x}\rho(\max(y\cap\bm x^+), y)&\text{otherwise},
\end{array}
\right.
$$
equivalently,
$$
r(\bm x, \bm y)=\max\limits_{y\in \bm y}\min\limits_{\begin{subarray}{l}x\in\bm x^+\\x\leq y\end{subarray}}\rho(x, y).
$$

Fix some of the simplest properties of the functions $d$ and $r$.

\begin{fact}\label{r_d_obv}
For all sets $\bm x, \bm y\subseteq X^-$,
$$
\bm x\subseteq \bm y\,\Rightarrow\, d(\bm x)\geq d(\bm y).
$$
For all finite sets $\bm x, \bm y, \bm z\subseteq X^-$,
$$\bm y\subseteq \bm z\,\Rightarrow\, r(\bm x, \bm y)\leq r(\bm x, \bm z).$$
\end{fact}
\begin{proof}
Immediately from the definitions. 
\end{proof}
Let us show that the function $r$ satisfies the \textit{triangle inequality}\footnote{However, the function $r$ is not a pseudo-metric since it is not symmetric.}. 
\begin{fact}\label{triangle} For all finite sets $\bm x, \bm y, \bm z\subseteq X^-$,
$$r(\bm x, \bm z)\leq r(\bm x, \bm y)+r(\bm y, \bm z).$$
\end{fact}

\begin{proof} If $\bm z\subseteq \bm x$, we have $r(\bm x, \bm z)=0$, the inequality holds. Consider the opposite case.  Let $c$ be an arbitrary
element of $\bm z\setminus \bm x$, and let $a=\max(c\cap\bm x^+)$. It suffices to prove that
$$
\rho(a,c)\leq r(\bm x, \bm y)+r(\bm y, \bm z).
$$
If $c\in\bm y$ we have: $\rho(a,c)\leq r(\bm x,
\bm y)\leq r(\bm x, \bm y)+r(\bm y, \bm z)$. Otherwise, denote 
$b=\max(c\cap\bm y^+)$. We have $\rho(b,c)\leq r(\bm y, \bm z)$. Consider the case $b\leq a$. Since $a<c$,  we have: $\rho(a,c)\leq\rho(b,c)\leq r(\bm y, \bm z)\leq
r(\bm x, \bm y)+r(\bm y, \bm z)$. Consider the opposite case. Suppose that $\bm x^+$ contains an element $x$ such that $a<x\leq b$. Since $b<c$, we have: $\max(c\cap\bm x^+)\geq x>a$, a contradiction. Therefore, $b\notin \bm x^+$ and  $a=\max(b\cap\bm x^+)$, which implies $\rho(a,b)\leq r(\bm x, \bm y)$. Consequently, $\rho(a,c)=\rho(a,b)+\rho(b,c)\leq r(\bm x, \bm
y)+r(\bm y, \bm z)$.
\end{proof}

\begin{fact}\label{existlocal} For any natural number $l\geq 1$, any set $\bm x\in [X^-]^n$ with the sparsity $d(\bm x)\geq n^{2l}$, and  any sets $\bm p, \bm q\in [2n]^n$ there is a set $\bm y\in [X^-]^n$ for which
\begin{enumerate}
\item $d(\bm y)\geq n^{2l-2}$,
\item $r(\bm x,\bm y)\leq n^{2l-1}$,
\item $\bm x\underset{\bm p\bm q}{\longrightarrow}\bm y$.
\end{enumerate}
\end{fact}

\begin{proof}
For each $i<n$, let $j_i=|\bm x_{[i]}\cap X|$. Thus, $\bm x_{[i]}=X_{[j_i]}$. For any $i\leq n$ define the set
$\bm z_i\subseteq X$ as follows:
\begin{align*}
&\bm z_0=\left\{X_{[kn^{2l-2}]}: 1\leq k\leq \bm p_{[0]}\right\},\\
&\bm z_i=\left\{X_{[j_{i-1}+kn^{2l-2}]}: 0\leq k\leq \bm p_{[i]}-\bm p_{[i-1]}-1\right\}\text{ for any $i\in \{1,2,\ldots, n-1\}$},\\
&\bm z_n=\left\{X_{[j_{n-1}+kn^{2l-2}]}: 0\leq k\leq 2n-1-\bm
p_{[n-1]}\right\}.
\end{align*}
\par Denote $\bm z=\bigcup\limits_{i\leq n}\bm z_i$. 
\par Note that for any $\bm p\in [2n]^n$ and $i<n$ we have
$$i\leq \bm p_{[i]}\leq n+i.$$ Consequently, $$\bm p_{[i]}-\bm p_{[i-1]}\leq n+1$$ for all $0<i<n$. Therefore, for all $i$, $1\leq i< n$, and also for $i=0$ if the set $\bm z_0$ is non-empty, we have:
$$
\rho\left(\min(\bm z_i), \max\left(\bm z_i\right)\right)\leq n\cdot n^{2l-2}=n^{2l-1},
$$
whence
$$
\max(\bm z_i)<\min(\bm z_{i+1}) \text{ and } \rho(\max(\bm z_i),\min(\bm z_{i+1}))\geq n^{2l}-n^{2l-1}
$$
(here we use what  $d(\bm x)\geq n^{2l}$). Now it is easy to check that:
\begin{enumerate}
\item[(a)] $|\bm z|=|\bm z_0|+|\bm z_1|+\ldots+|\bm z_n|=\bm p_{[0]}+(\bm p_{[1]}-\bm p_{[0]})+\ldots+(2n-\bm
p_{[n-1]})=2n$.
\item[(b)] For any $i<n$, $\bm x_{[i]}=\min(\bm z_{i+1})$, and
\begin{align*}
 |\bm x_{[i]}\cap \bm z|=|\bm z_0|+|\bm z_1|+\ldots+|\bm z_i|=\bm p_{[0]}+(\bm p_{[1]}-\bm p_{[0]})+\ldots+(\bm p_{[i]}-\bm p_{[i-1]})=\bm p_{[i]},
\end{align*}
i.e., $\bm x_{[i]}=\bm z_{[\bm p_{[i]}]}$. So, $\{\bm z_{[i]}: i\in \bm p\}=\bm x$.
\item[(c)] For all distinct $x, y\in \bm z^+$, $\rho(x,y)\geq \min\{n^{2l-2}, n^{2l}-n^{2l-1}\}=n^{2l-2}$, so $d(\bm z)\geq n^{2l-2}$.
\item[(d)] For all $z\in \bm z\setminus \bm x$, $\rho(\max(z\cap \bm x^+), z)\leq n^{2l-1}$, so $r(\bm x, \bm z)\leq n^{2l-1}$.
\end{enumerate}

\par Now it remains to put $\bm y=\{\bm z_{[i]}: i\in \bm q\}$ and use Fact \ref{r_d_obv}. 

\end{proof}

\begin{df}\label{Cascade_Def} For any natural numbers $l, t\geq 1$, a sequence $\bm x_0,\bm x_1, \ldots, \bm x_{l}$ of elements of~$[X^-]^n$ is called a $t$-\emph{cascade}  if for any $i<l$
\begin{enumerate}
\item $d(\bm x_i)\geq n^{t-2i}$,
\item $r(\bm x_{i},\bm x_{i+1})\leq n^{t-2i-1}$,
\item $\bm x_{i}\underset{Q}{\longleftrightarrow}\bm x_{i+1}$.
\end{enumerate}
A sequence $\bm x_0,\bm x_1, \ldots, \bm x_{l}$ is called a
\emph{cascade} if it is $t$-cascade for some natural number $t$.
\end{df}
\begin{fact}\label{exist} For any natural number $l\geq 1$, any set $\bm x\in [X^-]^n$ with the sparsity $d(\bm x)\geq n^{2l}$, and any sequence $(\bm p_0, \bm q_0),$ $(\bm p_1, \bm q_1),$ $ \ldots,$ $ (\bm p_{l-1}, \bm q_{l-1})$ of elements of $[2n]^n\times [2n]^n$ such that \mbox{$Q\subseteq \bigcap\limits_{i<l}Q_{\bm p_i\bm q_i}$} there is a $2l$-cascade $\bm x_0,\bm x_1, \ldots, \bm x_{l}$ for which $\bm x_0=\bm x$ and $\bm x_{i}\underset{\bm p_i\bm q_i}{\longrightarrow}\bm x_{i+1}$ for all $i<l$.
\end{fact}
\begin{proof} By induction on $l$ from Fact \ref{existlocal}.
\end{proof}

% Next, we need the following property of arbitrary $t$-cascades.

\begin{fact}\label{leftbord} Let $\bm x_0,\bm x_1, \ldots, \bm x_{l}$ be a $t$-cascade. Then for any $i$, $1\leq i\leq l$,
\begin{enumerate}
\item\label{leftbordItem1} $r(\bm x_{0},\bm x_{i})<\dfrac{n^t}{2}$,
\item\label{leftbordItem2} $\bm x_{0}\cap\bm x_{i} \subseteq  \bm x_{0}\cap \bm x_{i-1}$.
\end{enumerate}

\end{fact}
\begin{proof} Item \ref{leftbordItem1} follows from Fact \ref{triangle}:
\begin{align*}
&r(\bm x_{0},\bm x_{i})\leq r(\bm x_{0},\bm x_{1})+r(\bm x_{1},\bm
x_{2})+\ldots+r(\bm x_{i-1},\bm x_{i})\leq
n^{t-1}+n^{t-3}+\ldots+n^{t-2i+1}<\dfrac{n^t}{2}.
\end{align*}

\par Suppose that for some $a\in \bm x_0$
$$
a\in \bm x_{i}\text{ and }a\notin \bm x_{i-1}.
$$
Let $b=\max(a\cap \bm x_{i-1}^+)$. So, $\rho(b,a)\leq r(\bm x^{i-1}, \bm
x^{i})\leq n^{t-2i+1}$. Since $d(\bm x_0)\geq n^t$, we have: $b\notin \bm x_0^+$
and, therefore, $i-1\geq 1$. Let $c=\max(b\cap \bm x_0^+)$. By Item \ref{leftbordItem1} we have:  $\rho(c,b)\leq r(\bm x_0, \bm
x_{i-1})<
\dfrac{n^t}{2}$. Thus, $\rho (c, a)= \rho(c,b)+\rho(b,a)<
n^{t-2i+1}+\dfrac{n^t}{2}<n^t$, which contradicts to $d(\bm x_0)\geq
n^t$. Item \ref{leftbordItem2} is proved.
\end{proof}

\par For each $\bm p, \bm q\in [2n]^n$  define the function \mbox{$\varphi_{\bm p\bm q}\subset n\times n$}: $\mathrm{dom}\,\varphi_{\bm p\bm q}=\{i<n: \bm p_{[i]}\in \bm q\}$, and $\varphi_{\bm p\bm q}(i)=|\bm p_{[i]}\cap \bm q|$ for any $i\in \mathrm{dom}\,\varphi_{\bm p\bm q}$. Thus, for all $i, j<n$,
$$
\bm p_{[i]}=\bm q_{[j]}\text{ iff } i\in \mathrm{dom}\, \varphi_{\bm p\bm q}\text{ and }\varphi_{\bm p\bm q}(i)=j.
$$
Identifying each set $\bm s\in [2n]^n$ with the function $f_{\bm s}:n\to 2n$, $f_{\bm s}(i)=\bm s_{[i]}$, we can simply write $\varphi_{\bm p\bm q}=\bm q^{-1}\circ\bm p$.

\begin{fact}\label{Varphi} Let $\bm p, \bm q\in [2n]^n$, $\bm x, \bm y\in [X]^n$, $\bm x\underset{\bm p\bm q}{\longrightarrow}\bm y$, and $i, j<n$. Then
$$
\bm x_{[i]}=\bm y_{[j]}\text{ iff } i\in \mathrm{dom}\, \varphi_{\bm p\bm q}\text{ and }\varphi_{\bm p\bm q}(i)=j.
$$

\end{fact}
\begin{proof}  Arguing as in the proof of Fact \ref{I_pq_stabl}, we have the following chain of equivalences:
$$
\bm x_{[i]}=\bm y_{[j]}\,\Leftrightarrow\, \bm z_{[\bm p_{[i]}]}=\bm z_{[\bm q_{[j]}]}\,\Leftrightarrow\, \bm p_{[i]}= \bm q_{[j]}\,\Leftrightarrow\, (i\in \mathrm{dom}\, \varphi_{\bm p\bm q}\text{ and }\varphi_{\bm p\bm q}(i)=j).
$$
\end{proof}

\begin{df}\label{Full_Cascad_Def} A cascade $\bm x_0\underset{\bm p_0\bm q_0}{\longrightarrow}\bm x_1\underset{\bm p_1\bm q_1}{\longrightarrow} \ldots\underset{\bm p_{l-1}\bm q_{l-1}}{\longrightarrow}\bm x_l$ is \textit{full} if
\begin{enumerate}
\item for all distinct $\bm p, \bm q\in [2n]^n$ such that $Q\subseteq Q_{\bm p\bm q}$ there is a number $i<l$ for which $(\bm p_i,\bm q_i)=(\bm p, \bm q)$ or $(\bm p_i,\bm q_i)=(\bm q, \bm p)$;
\item for any natural number $j<l$ there are natural numbers $j_0, j_1$  for which
\begin{enumerate}
\item $j_0\leq j\leq j_1<l$
\item $j_1-j_0\geq n-1$,
\item $(\bm p_{j_0}, \bm q_{j_0})=(\bm p_{j_0+1}, \bm q_{j_0+1})=\ldots=(\bm p_{j_1}, \bm q_{j_1})$.
\end{enumerate}
\end{enumerate}
\end{df}

\begin{fact}\label{extrusion} Let $\bm x_0,\bm x_1, \ldots,\bm x_l$ be a full cascade. Then for any $i<n$
$$
(\bm x_0)_{[i]}\in \bm x_l\,\Leftrightarrow\, i\in I(Q).
$$
\end{fact}
%\textbf{++++ REFORMULATE AS $x\in \bm x_0\cap x_1\,\Rightarrow\, x\cap \bm x_0\in I(Q)$? ++++++++}
\begin{proof}
The implication $i\in I(Q) \,\Rightarrow\, (\bm x_0)_{[i]}\in \bm x_l$ follows from Fact \ref{I_pq_stabl}. Let as prove the reverse implication.
Suppose that for some number $i<n$
$$
i\notin I(Q)\text{ and } (\bm x_0)_{[i]}\in \bm x_l.
$$
From Fact \ref{leftbord}, Item \ref{leftbordItem2}, we have: $(\bm x_0)_{[i]}\in \bm x_k$ for all $k\leq l$. For all $k\leq l$, denote
$$
\theta(k)=|(\bm x_0)_{[i]}\cap \bm x_k|.
$$
Thus, $\theta(k)\in n$ and $(\bm x_0)_{[i]}=(\bm x_k)_{[\theta(k)]}$. Let
$
\bm x_0\underset{\bm p_0\bm q_0}{\longrightarrow}\bm x_1\underset{\bm p_1\bm q_1}{\longrightarrow} \ldots\underset{\bm p_{l-1}\bm q_{l-1}}{\longrightarrow}\bm x_l
$.
From Fact \ref{Varphi} we have:
$$
\theta(k)=\varphi_{\bm p_{k-1}\bm q_{k-1}}\circ \varphi_{\bm p_{k-2}\bm q_{k-2}}\circ\ldots\circ\varphi_{\bm p_0\bm q_0}(i).
$$
Since $I_{\bm p\bm q}=I_{\bm q\bm p}$ for all $\bm p,\bm q\in [2n]^n$ and the given cascade is full, there is a number $j<l$ such than $i\notin I_{\bm p_j\bm q_j}$. Let us choose the minimum of these numbers. Thus, $i\notin I_{\bm p_j\bm q_j}$ and $i\in I_{\bm p_k\bm q_k}$ for all $k<j$. From Fact~\ref{I_pq_stabl} we have
$$
\theta(j)=\theta(j-1)=\ldots=\theta(0)=i.
$$
Besides, since the given cascade is full, we have:
$$
(\bm p_{j}, \bm q_{j})=(\bm p_{j+1}, \bm q_{j+1})=\ldots=(\bm p_{j+s}, \bm q_{j+s}).
$$
for some $s\geq n-1$. Therefore, for all $k$, $j< k\leq j+s+1$, we have
$$
\theta(k)=\underbrace{\varphi_{\bm p_j\bm q_j}\circ
\varphi_{\bm p_j\bm q_j}\circ\ldots\circ\varphi_{\bm p_j\bm q_j}}_{\text{$k-j$
times}}(i).
$$
It is easy to see that
$$
x<y\,\Rightarrow\, \varphi_{\bm p\bm q}(x)<\varphi_{\bm p\bm q}(y).
$$
for all $\bm p, \bm q\in [2n]^n$ and $x,y\in \mathrm{dom}\, \varphi_{\bm p\bm q}$. Consequently, the
sequence $\theta(k)$, $j\leq k\leq j+s+1$, monotonically increases (if
$\varphi_{\bm p_j\bm q_j}(i)>i$), monotonically decreases  (if
$\varphi_{\bm p_j\bm q_j}(i)<i$), or is constant  (if
$\varphi_{\bm p_j\bm q_j}(i)=i$). 
All these cases lead to a contradiction. The first two cases imply $\theta(j+s+1)\notin n$, and the last case is equivalent to
$i\in I_{\bm p_j\bm q_j}$. 
\end{proof}

\begin{fact}\label{existrepetfreeall} For any set \mbox{$\bm x\in [X]^n$} with the sparsity $d(\bm x)\geq m= n^{n\binom{2n}{n}\left(\binom{2n}{n}-1\right)}$ there is an \mbox{$n\binom{2n}{n}\left(\binom{2n}{n}-1\right)$}-cascade $\bm x_0, \bm x_1, \ldots, \bm x_{l}$ such that $\bm x_0=\bm x$, and
$\bm x_0\cap \bm x_{l}=\{(\bm x_0)_{[i]}: i\in I(Q)\}.$
\end{fact}
\begin{proof}
By combining Facts \ref{exist} and \ref{extrusion}.
\end{proof}

Now we will prove that the equivalence class $[\bm
x]_{\underset{Q}{\longleftrightarrow}}$ of a sufficiently sparse set $\bm x$ is
closed with respect to ``small shifts'' of elements $\bm x_{[i]}$,
$i\notin I(Q)$.

\begin{fact}\label{changelocal} Let $\bm p, \bm q\in [2n]^n$, $\bm x, \bm y\in [X^-]^n$,  $\bm x\underset{\bm p\bm q}{\longrightarrow}\bm y$, $i<n$, and $\bm x_{[i]}\notin \bm y$. Let $a=\max(\bm x_{[i]}\cap (\bm x\cup\bm y)^+)$. Then  for any $x\in X$ such that $a<x\leq \bm x_{[i]}$ and $\rho(a,x)\geq 2n$,
$$
(\bm x\setminus \{\bm x_{[i]}\})\cup \{x\}\underset{\bm p\bm q}{\longrightarrow}\bm y,
$$
and, therefore, $(\bm x\setminus \{\bm x_{[i]}\})\cup \{x\}\underset{Q}{\longleftrightarrow}\bm x$.
\end{fact}
\begin{proof} Let $\bm z\in [X]^{2n}$, $\{\bm z_{[i]}: i\in \bm p\}=\bm x$, and $\{\bm z_{[i]}: i\in \bm q\}=\bm y$.  Let the set $\bm z_0$ and the numbers $j$, $s$, $k$ be such that $\bm z_0=\{z\in \bm z: a\leq z\leq \bm x_{[i]}\}=\{\bm z_{[j]}, \bm z_{[j+1]}, \ldots, \bm z_{[j+s]}\}$, and $\bm z_{[j]}=X_{[k]}$. It is easy to see that $1\leq |\bm z_0|< 2n$, and for any $b\in\bm x\cup \bm y$, $b\leq \min(\bm z_0)$ or $b\geq \max(\bm z_0)$. Let
$$
\bm z_0'=\{X_{[k]}, X_{[k+1]}, \ldots, X_{[k+s-1]}, x\}\text{ and }\bm z^\ast=(\bm z\setminus \bm z_0)\cup \bm z_0'.
$$
Since $\rho(a, x)\geq 2n$, we have $x> X_{[k+s-1]}$ (or $\bm z_0=\{\bm x_{[i]}\}$), $|\bm z_0'|=|\bm z_0|$, $(\bm z\setminus \bm z_0)\cap \bm z_0'=\emptyset$, and for any $b\in(\bm x\cup \bm y)\setminus \{\bm x_{[i]}\}$, $b\leq \min(\bm z'_0)$ or $b> \max(\bm z'_0)$. So, $|\bm z^\ast|=2n$, and for any $b\in (\bm x\cup \bm y)\setminus \{\bm x_{[i]}\}$,
$$
|b\cap \bm z^\ast|=|b\cap (\bm z\setminus \bm z_0)|+|b\cap \bm z_0'|=|b\cap \bm z|.
$$ 
Besides, $|\bm x_{[i]}\cap \bm z|=|x\cap \bm z^\ast|$. Consequently,
\begin{align*}
\{\bm z^{\ast}_{[i]}:i\in \bm p\}=(\bm x\setminus \{\bm x_{[i]}\})\cup \{x\},
\text{ and } \{\bm z^{\ast}_{[i]}:i\in \bm q\}=\bm y.
\end{align*}
The Fact is proved.
\end{proof}

\begin{fact}\label{change} Let $\bm x\in [X^-]^n$, $i\in n\setminus I(Q)$, $x\in X$, $\bm y=(\bm x\setminus \{\bm x_{[i]}\})\cup\{x\}$, $\min(d(\bm x), d(\bm y))\geq m$, and either $\bm x_{[i-1]}<x\leq \bm x_{[i]}$, or $i=0$ and $e<x\leq \bm x_{[0]}$. 
Then $
\bm x\underset{Q}{\longleftrightarrow}\bm y
$.
\end{fact}

\begin{proof} Notice that $\dfrac{m}{2}=\dfrac{1}{2}n^{n\binom{2n}{n}\left(\binom{2n}{n}-1\right)}\geq 2n$ (we assume all the time that $n\geq 2$). By Fact \ref{existrepetfreeall} there exists an \mbox{$n\binom{2n}{n}\left(\binom{2n}{n}-1\right)$}-cascade $\bm x_0, \bm x_1, \ldots, \bm
x_{l}$ such that $\bm x_0=\bm x$ and $\bm x_{[i]}\notin \bm x_{l}$. Since the relation $\underset{Q}{\longleftrightarrow}$ is transitive, we have $\bm x\underset{\bm p\bm q}{\longrightarrow} \bm x_l$ for some $\bm p, \bm q\in [2n]^n$ for which $Q\subseteq Q_{\bm p\bm q}$. 
Let $a=\max(\bm x_{[i]}\cap (\bm x\cup\bm x_l)^+)$, and let
$$
b=\max(\bm x_{[i]}\cap \bm x^+)=\begin{cases}
\bm x_{[i-1]}& \text{if $i\neq 0$},\\
e & \text{otherwise}.
\end{cases}
$$
Obviously, $b\leq a<\bm x_{[i]}$. Since $d(\bm y)\geq m$, we have $\rho(b, x)\geq m$.  If $a\in \bm x^+$ then $b=a$, and we can immediately use Fact \ref{change}. Otherwise,  $\rho(b,a)\leq r(\bm x, \bm y) < \dfrac{m}{2}$ by
Item \ref{leftbordItem1} of Fact \ref{leftbord}. Therefore,
$\rho(a,x)=\rho(b,x)-\rho(b,a)\geq \dfrac{m}{2}\geq 2n$. It remains to apply Fact \ref{change} again.  
\end{proof}
We can now complete the proof of Lemma  \ref{mainlemma}.  For any $\bm x\in [X^-]^n$ with the sparsity $d(\bm x)\geq m$ define the set~$\widehat{\bm x}$ as follows: 
\begin{enumerate}
\item $\widehat{\bm x}_{[i]}=\bm x_{[i]}$ for all $i\in I(Q)$,
\item if $0\notin I(Q)$ then $\widehat{\bm x}_{[0]}=X_{[m]}$,
\item for all $i$, $0<i<n$, if $i\notin I(Q)$ then $\rho(\widehat{\bm x}_{[i-1]}, \widehat{\bm x}_{[i]})=m$.
\end{enumerate}
It is easy to see that $\widehat{\bm x}_{[i]}\leq \bm x_{[i]}$ for all
$i<n$. For each $i\leq n$, let
$$
\widehat{\bm x}_i=\{\widehat{\bm x}_{[0]}, \widehat{\bm x}_{[1]}, \ldots, \widehat{\bm x}_{[i-1]}, \bm x_{[i]}, \bm
x_{[i+1]}, \ldots, \bm x_{[n-1]}\}.
$$
Obviously, for all $i< n$,
$$
\widehat{\bm x}_{i+1}= \left(\widehat{\bm x}_{i}\setminus \left(\widehat{\bm x}_{i}\right)_{[i]}\right)\cup \{\widehat{\bm x}_{[i]}\}.
$$
Besides, $d(\widehat{\bm x}_i)\geq m$ for all $i\leq n$. By Fact \ref{change}, we have
$$
\bm x=\widehat{\bm x}_0\underset{Q}{\longleftrightarrow}\widehat{\bm x}_1\underset{Q}{\longleftrightarrow}\ldots\underset{Q}{\longleftrightarrow}\widehat{\bm x}_n=\widehat{\bm x}.
$$
It remains to note that $\widehat{\bm x}=\widehat{\bm y}$ for all $\bm x, \bm y\in [X]^n$ such that $\min(d(\bm x), d(\bm y))\geq m$ and $\bm x_{[i]}=\bm y_{[i]}$ for all $i\in I(Q)$.
\end{proof}
Theorem \ref{main} easily follows from Lemma \ref{mainlemma}.
Let
$$
R_0=\{X_{[i]}: i<m\} \text{ and } R_j = \{X_{[j+is]}: 1\leq i<\omega\}\text{ for all j, $1\leq j\leq m$}.
$$
The family $\{R_j\}_{j\leq m}$ is a partition of $X$. The set $R_0$ is finite of cardinality $m$. The sparsity of $R_j$, $1\leq j\leq m$, is $m$. Therefore, $d(\bm x)\geq m$ for all $\bm x\in [R_j]^n$. By Lemma \ref{mainlemma}, for all $j$, $1\leq j\leq m$, and  $\bm x, \bm y\in [R_j]^n$, we have  
$$
\{\bm x_{[i]}:i\in I(Q)\}=\{\bm y_{[i]}:i\in I(Q)\}\,\Rightarrow\, \bm
x\approx_\mathcal P\bm y,
$$
i.e., $R_j$ is $I(Q)$-canonical for $\mathcal P$.
\end{proof}

\subsection*{2. Application to  the Theory of Ultrafilters}
\par An ultrafilter $\mathfrak u$ on $\mathscr{P}(\omega)$ is called a \emph{Ramsey ultrafilter} if it is nonprincipal and for any $n$, $1\leq n<\omega$, and finite partition~$\mathcal P$ of $[\omega]^n$, $\mathfrak u$ contains some set $X\subseteq \omega$ that is homogenous for $\mathcal P$.
Continuum hypothesis (and some other assumptions, including Martin's Axiom) implies the existence of Ramsey ultrafilters.\footnote{However, the existence of Ramsey ultrafilters is independent of ZFC, see \cite{Wimmers}.} There are many equivalent characterizations of Ramsey ultrafilters, see \cite{Comfort Negrepontis}. In particular, an ultrafilter~$\mathfrak u$ is Ramsey if and only if it is selective, and if and only if it is minimal.
\par Recall these definitions. An ultrafilter $\mathfrak u$ on $\mathscr{P}(\omega)$ is \emph{selective} if for every function \mbox{$f:\omega\to \omega$} there is $X\in \mathfrak u$ such that the restriction $f\upharpoonright X$ of $f$ to $X$ is either one-to-one or constant.
\par The concept of a minimal ultrafilter is based on the notion of ultrafilter extension of unary functions and the Rudin-Keisler (pre)order. For any set $A$, the set of all ultrafilters on $\mathscr P(A)$ is denoted by $\bm\beta A$. For any function $f: A\to B$ the ultrafilter extension $\widetilde f$ of $f$ is the function from $\bm\beta A$ to $\bm\beta B$ defined by
$$
\widetilde f(\mathfrak u)=\{S\subseteq Y: (\forall X\in \mathfrak u)(\exists x\in X)\, f(x)\in S\}
$$
for all $\mathfrak u\in \bm\beta A$.
\par The \textit{Rudin-Keisler preorder} on $\bm\beta A$ is the binary relation $\leq_{\mathrm{RK}}$ defined by
$$
\mathfrak u\leq_{\mathrm{RK}}\mathfrak v\,\Leftrightarrow\, \widetilde f(\mathfrak v)=\mathfrak u\text{ for some $f:A\to A$}
$$
for all $\mathfrak u, \mathfrak v\in \bm\beta A$. An ultrafilter $\mathfrak u\in \bm\beta A$ is called \emph{minimal} if
$$
\mathfrak v\leq_{\mathrm{RK}}\mathfrak u\,\Rightarrow\,\text{ $\mathfrak v$ is principal or $\mathfrak u\leq_{\mathrm{RK}}\mathfrak v$}
$$
for any $\mathfrak v\in \bm\beta A$.
In other words, $\mathfrak u$ is minimal if for any function $f: A\to A$ either $\widetilde f(\mathfrak u)$ is principal or there is a function $g:A\to A$ such that $\widetilde g\left(\widetilde f (\mathfrak u)\right)=\mathfrak u$.
\par The equivalence relation $\leq_{\mathrm{RK}}\cap \leq^{-1}_{\mathrm{RK}}$ is denoted by $\approx_{\mathrm{RK}}$. The Rudin-Keisler preorder naturally extends to the quotient set $\bm\beta A/\approx_{\mathrm{RK}}$: $\tau(\mathfrak u)\leq_{\mathrm{RK}} \tau(\mathfrak v)\,\Leftrightarrow\, \mathfrak u\leq_{\mathrm{RK}}\mathfrak v$ for all equivalence class $\tau(\mathfrak u)$ and $\tau(\mathfrak v)$ of ultrafilters $\mathfrak u$ and $ \mathfrak v$, respectively. The relation $\leq_{\mathrm{RK}}$ is a (partial) order on $\bm\beta A/\approx_{\mathrm{RK}}$, and ultrafilter $\mathfrak u$ is minimal iff the equivalence class $\tau(\mathfrak u)$ is a minimal element of the poset $\left(\bm\beta A/\approx_{\mathrm{RK}}\right)\setminus \{\tau(a)\}$ where $a$ is any principal ultrafilter on $A$ (all principal ultrafilters on $A$ are equivalent w.r.t. $\approx_{\mathrm{RK}}$).
\par  Theorem \ref{main} allows us to propose a modification of the above characterizations of Ramsey ultrafilters using $n$-ary maps and their  ultrafilter extensions.

\par Ultrafilter extensions of binary maps, especially of group and semigroup operations, have been considered since the 60s of the 20th century. The results obtained in this field have found numerous Ramsey-theoretic
applications in number theory, algebra, topological
dynamics, and ergodic theory.
The book~\cite{Hindman Strauss} is a~comprehensive
treatise of this area, with an historical information.
\par Ultrafilter extensions of arbitrary $n$-ary maps (and, more broadly, of first-order models)
have been introduced independently in recent works
by Goranko~\cite{Goranko} and
Saveliev~\cite{Saveliev,Saveliev(inftyproc)}. Further studies  can be found in \cite{Saveliev(idempotents),Poliakov Saveliev, Poliakov Saveliev 2019, Saveliev Shelah}.

 For
a~map $f:A^n\to B$, the extended
map $\widetilde f:(\bm\beta A)^n\to\bm\beta B$
can be defined by recursion. A nullary function $f$ is identified with a constant $c_f\in B$. For $n=0$, we define $\widetilde f$ as the principal ultrafilter generated by $c_f$, i.e. $\widetilde f=\{S\subseteq B: c_f\in S\}$. For $n>0$ we define $$\widetilde f(\mathfrak u_1, \mathfrak u_2, \ldots, \mathfrak u_n)=\{S\subseteq B: (\forall X\in \mathfrak u_1)(\exists x\in X)\, S\in \widetilde f_x(\mathfrak u_2, \ldots, \mathfrak u_n)\},$$ where $f_x(x_2, \ldots, x_n)=f(x, x_2, \ldots, x_n)$ for all $x, x_2, \ldots, x_n\in A$. It is easy to verify that for $n = 1$ we have the definition equivalent to that given above.
\par For all bijections $f,g:A^n\to A$ there is a function $h:A\to A$ such that $f(x_0, x_1, \ldots, x_{n-1})=h(g(x_0, x_1, \ldots, x_{n-1}))$ for all $x_0, x_1, \ldots, x_{n-1}\in A$. In \cite{Saveliev},  it is proved that the extension operator commutes with the composition $h\circ g$ if $h$ is an one-place function. So, we have
$$
\widetilde f(\mathfrak u_0, \mathfrak u_1, \ldots, \mathfrak u_{n-1})=\widetilde{h\circ g}(\mathfrak u_0, \mathfrak u_1, \ldots, \mathfrak u_{n-1})=\widetilde{h}(\widetilde{g}(\mathfrak u_0, \mathfrak u_1, \ldots, \mathfrak u_{n-1})
$$
for all $\mathfrak u_0, \mathfrak u_1, \ldots, \mathfrak u_{n-1}\in\bm\beta A$. Therefore, ultrafilters $\widetilde f(\mathfrak u_0, \mathfrak u_1, \ldots, \mathfrak u_{n-1})$ and $\widetilde g(\mathfrak u_0, \mathfrak u_1, \ldots, \mathfrak u_{n-1})$ are $\mathrm{RK}$-equivalent. Considering ultrafilters up to equivalence relation~$\approx_\mathrm{RK}$, we denote by $\mathfrak u_0\times\mathfrak u_1\times\ldots\times\mathfrak u_{n-1}$ the ultrafiter $\widetilde f(\mathfrak u_0, \mathfrak u_1, \ldots, \mathfrak u_{n-1})$ for some one-to-one map $f: A^n\to A$.
\begin{df} A function $f:\omega^n\to\omega$ is called \emph{selectively upward injective on a set $X\subseteq \omega$ w.r.t. a set (of indices) $I\subseteq n$} if
$$
f(x_0, x_1, \ldots, x_{n-1})=f(y_0, y_1, \ldots, y_{n-1})\,\Leftrightarrow\, \bigwedge\limits_{i\in I}(x_i=y_i)
$$
for all $x_0<x_1< \ldots< x_{n-1}$ and $y_0<y_1< \ldots< y_{n-1}$ from $X$.  A function $f:\omega^n\to\omega$ is called
\begin{enumerate}
\item[i.] \emph{selectively upward injective} on a set $X\subseteq \omega$ if it is selectively injective on a set $X\subseteq \omega$ w.r.t. some non-empty set of indices $J\subseteq n$,
\item[ii.] \emph{upward constant} on a set $X\subseteq \omega$ if it is selectively injective on a set $X\subseteq \omega$ w.r.t. $\emptyset$, i.e., $$
f(x_0, x_1, \ldots, x_{n-1})=f(y_0, y_1, \ldots, y_{n-1})
$$
for all $x_0<x_1< \ldots< x_{n-1}$ and $y_0<y_1< \ldots< y_{n-1}$ from $X$.
\end{enumerate}
\end{df}
\begin{thm}\label{Th2} Let $\mathfrak u$ be a non-principal ultrafilter on $\mathscr P(\omega)$. Then the following conditions are equivalent:
\begin{enumerate}
\item $\mathfrak u$ is Ramsey ultrafilter;
\item for every $n$, $1\leq n<\omega$, and partition~$\mathcal P$ of $[\omega]^n$, $\mathfrak u$ contains some set $X$ that is canonical for $\mathcal P$;
\item for every $n$, $1\leq n<\omega$, and function \mbox{$f:\omega^n\to \omega$}, $\mathfrak u$ contains some set $X$ such that~ $f$ is either selectively upward injective or upward constant on $X$;
%\item For every natural number $n\geq 1$ and function $f:\omega^n\to \omega$ there is a set $X\in \mathfrak u$ such that $f$ is either upper quasi-invertible on $X$ or upper constant on $X$.
\item for every $n$, $1\leq n<\omega$, and function $f:\omega^n\to \omega$, either $
\widetilde f(\mathfrak u, \mathfrak u, \ldots, \mathfrak u)$ is principal or $\widetilde f(\mathfrak u, \mathfrak u, \ldots, \mathfrak u)\approx_\mathrm{RK}\underbrace{\mathfrak u\times \mathfrak u\times\ldots\times\mathfrak u}_{\text{$m$ times}}$ for some $m$, $1\leq m\leq n$.

%\item For every natural number $n\geq 1$ and function $f:\omega^n\to \omega$
%$$
%\widetilde f(\mathfrak u, \mathfrak u, \ldots, \mathfrak u)\text{ is principal or $\mathfrak u\leq_\mathrm{RK}\widetilde f(\mathfrak u, \mathfrak u, \ldots, \mathfrak u)$}
%$$
\end{enumerate}
\end{thm}
\begin{proof} ($1\,\Rightarrow\, 2$). Let $\mathcal P$ be a partition of $[\omega]^n$. By Theorem \ref{main_var} there is a finite partition $\mathcal Q$ of $[\omega]^{2n}$ such that any set $X\subseteq \omega$ that is homogeneous for $\mathcal Q$ is a finite union of sets $X_0, X_1, \ldots, X_m$ that are canonical for $\mathcal P$. Since $\mathfrak u$ is a Ramsey ultrafilter, it contains some set $X$ that is homogeneous for $\mathcal Q$. Since $\mathfrak u$ is an ultrafilter, it contains one of the sets $X_0, X_1, \ldots, X_m$.

\par ($2\,\Rightarrow\, 3$). For any $c\in \omega$ denote $P_c=\{\bm x\in [\omega]^n: f(\bm x_{[0]}, \bm x_{[1]}, \ldots, \bm x_{[n-1]})=c\}$. Obviously, the set $\mathcal P=\{P_c: c\in \omega\}$ is a partition of $[\omega]^n$, and a set $X\subseteq \omega$ is $I$-canonical for $\mathcal P$ if and only if $f$ is selectively upward injective on $X$ w.r.t. $I$.

\par ($3\,\Rightarrow\, 4$). First, let us prove the following proposition.
\begin{prop}\label{Prop1} Let $\mathfrak u\in \bm\beta\omega\setminus\omega$, $n,m\in \omega$, \mbox{$f:\omega^n\to\omega$}, $g:\omega^m\to\omega$. Let also $k\in \omega$, $\bm p\in [k]^n$, $\bm q\in [k]^m$, and there is a set $X\in \mathfrak u$ such that
$$
f\left(x_{\bm p_{[0]}}, x_{\bm p_{[1]}}, \ldots, x_{\bm
p_{[n-1]}}\right)=g\left(x_{\bm q_{[0]}}, x_{\bm q_{[1]}}, \ldots,
x_{\bm q_{[m-1]}}\right).
$$
for all $x_0<x_1<\ldots <x_{k-1}\in X$. Then

$$
\widetilde f(\underbrace{\mathfrak u, \mathfrak u, \ldots, \mathfrak u}_{\text{$n$ times}})=\widetilde g(\underbrace{\mathfrak u, \mathfrak u, \ldots, \mathfrak u}_{\text{$m$ times}}).
$$
\end{prop}
\begin{proof} By induction on $n+m$. The case $n=m=0$ (induction base) is clear.
\par Let $n+m>0$, and
$$
f\left(x_{\bm p_{[0]}}, x_{\bm p_{[1]}}, \ldots, x_{\bm
p_{[n-1]}}\right)=g\left(x_{\bm q_{[0]}}, x_{\bm q_{[1]}}, \ldots,
x_{\bm q_{[m-1]}}\right).
$$
for all $x_0<x_1<\ldots <x_{k-1}\in X$. Without loss of generality,
suppose $n>0$ and $m>0\,\Rightarrow\,\bm p_{[0]}\leq \bm q_{[0]}$.
Denote $\bm p'=\bm p\setminus\{\bm p_{[0]}\}$. For all $y\in X$ and
$x_0<x_1<\ldots<x_{k-1}\in X\setminus (y+1)$ we have
$$
f_y\left(x_{\bm p'_{[0]}}, x_{\bm p'_{[1]}}, \ldots, x_{\bm
p'_{[n-2]}}\right)=
\begin{cases}
g\left(x_{\bm q_{[0]}}, x_{\bm q_{[1]}}, \ldots, x_{\bm q_{[m-1]}}\right)& \text{ if $m= 0$ or $\bm p_{[0]}< \bm q_{[0]}$},\\
g_y\left(x_{\bm q_{[0]}}, x_{\bm q_{[1]}}, \ldots, x_{\bm
q_{[m-1]}}\right)&\text{ if $m\neq 0$ and $\bm p_{[0]}= \bm q_{[0]}$}.
\end{cases}
$$
Since ultrafilter $\mathfrak u$ is non-principal, $X\setminus(y+1)\in \mathfrak u$. Hence, by induction hypothesis, for any $y\in X$ we have
$$
\widetilde f_y(\mathfrak u, \ldots, \mathfrak u)=
\begin{cases}
\widetilde g(\mathfrak u,\mathfrak u, \ldots, \mathfrak u)& \text{ if $m= 0$ or $\bm p_{[0]}< \bm q_{[0]}$},\\
\widetilde g_y(\mathfrak u,\mathfrak u \ldots, \mathfrak u)& \text{
if $m\neq 0$ and $\bm p_{[0]}= \bm q_{[0]}$}.
\end{cases}
$$
Let $S\in \widetilde f(\mathfrak u, \mathfrak u, \ldots, \mathfrak u)$, i.e., $$ (\forall Y\in \mathfrak u)(\exists y\in Y)\, S\in \widetilde f_y(\mathfrak u, \ldots, \mathfrak u).$$ In both the cases, $S\in \widetilde g(\mathfrak u,\mathfrak u, \ldots, \mathfrak u)$, which implies $\widetilde f(\mathfrak u, \mathfrak u, \ldots, \mathfrak u)\subseteq \widetilde g(\mathfrak u, \mathfrak u, \ldots, \mathfrak u)$. Since $\widetilde f(\mathfrak u, \mathfrak u, \ldots, \mathfrak u)$ and $\widetilde g(\mathfrak u, \mathfrak u, \ldots, \mathfrak u)$ are ultrafilters, we have $\widetilde f(\mathfrak u, \mathfrak u, \ldots, \mathfrak u)= \widetilde g(\mathfrak u, \mathfrak u, \ldots, \mathfrak u)$.
\end{proof}
\par We continue the proof of the implication $2\Longrightarrow 3$. Let a function $f:\omega^n\to \omega$ be selectively upward injective on $Y\in \mathfrak u$ w.r.t. $I\subseteq n$. Denote $|I|=m$. Let $Z$ be the set of all sequences $(x_0, x_1, \ldots, x_{m-1})\in Y^m$ such  that $x_0<x_1<\ldots <x_{m-1}$, $|X\cap x_0|\geq I_0$, and $|X\cap (x_{i}\setminus x_{i-1})|\geq I_{[i]}-I_{[i-1]}$ for all $i$, $1\leq i\leq m-1$. Define the function $g_0: Z\to\omega$ by
$$
g_0(x_0, x_1, \ldots, x_{m-1})=f(y_0, y_1, \ldots, y_{n-1})
$$
for some $(y_0, y_1, \ldots, y_{n-1})\in X^n$ for which $y_0<y_1<\ldots <y_{n-1}$ and $y_{I_{[i]}}=x_i$ for all $i<m$. The function $g_0$ is well defined because 
$$
\bigwedge\limits_{i\in J}(x_i=y_i)\,\Rightarrow\, f(x_0, x_1, \ldots, x_{n-1})=f(y_0, y_1, \ldots, y_{n-1})
$$
for all $x_0<x_1< \ldots< x_{n-1}$ and $y_0<y_1< \ldots< y_{n-1}$ from $X$. Besides, since for all $x_0<x_1< \ldots< x_{n-1}$ and $y_0<y_1< \ldots< y_{n-1}$ from $X$
 $$
f(x_0, x_1, \ldots, x_{n-1})=f(y_0, y_1, \ldots, y_{n-1})\,\Rightarrow\, \bigwedge\limits_{i\in J}(x_i=y_i),
$$
the function $g_0$ is injective, or $m=0$ and $g_0$ is constant.
\par If $m=0$ the ultrafilter $\widetilde f(\mathfrak u, \mathfrak u, \ldots, \mathfrak u)$ is principal by Proposition \ref{Prop1} .
\par Let $m>0$. Choose a set $Z'\subseteq \omega$ such that $|Z'|=|\omega\setminus Z'|=\omega$ and functions $h_1, h_2:\omega\to \omega$ such that $h_1$ bijectively maps $g(Z)$ onto $Z'$ and $h_2(h_1(x))=x$ for all $x\in Z'$. The map $h_1\circ g: Z\to Z'$ can be extended to a one-to-one function $w: \omega^m\to\omega$. By Proposition \ref{Prop1} we have
\begin{align*}
\widetilde h_1(\widetilde f(\mathfrak u, \mathfrak u, \ldots, \mathfrak u))=\widetilde{h_1\circ f}(\mathfrak u, \mathfrak u, \ldots, \mathfrak u)=\widetilde{w}(\mathfrak u, \mathfrak u, \ldots, \mathfrak u),\\
\widetilde f(\mathfrak u, \mathfrak u, \ldots, \mathfrak u)=\widetilde{h_2\circ w}(\mathfrak u, \mathfrak u, \ldots, \mathfrak u)=\widetilde h_2(\widetilde w(\mathfrak u, \mathfrak u, \ldots, \mathfrak u)).
\end{align*}
Therefore, $
\widetilde f(\mathfrak u, \mathfrak u, \ldots, \mathfrak u)\approx_\mathrm{RK}\underbrace{\mathfrak u\times \mathfrak u\times\ldots\times\mathfrak u}_{\text{$m$ times}}.
$
\par To prove $3\,\Rightarrow\, 1$, it suffices to restrict ourselves to the case $n=1$ and recall that each minimal ultrafilter is a Ramsey ultrafilter, see \cite{Comfort Negrepontis}.
\end{proof}

\paragraph*{Remark.}
\par In combinatorial applications of the theory of ultrafilters, non-principal \emph{idempotents} are of great importance, see \cite{Hindman Strauss,Saveliev(idempotents)}. It is well known that among Ramsey ultrafilters~$\mathfrak u$ there are no one such that $\mathfrak u+ \mathfrak u=\mathfrak u$ or $\mathfrak u\cdot \mathfrak u=\mathfrak u$. It can be shown that this property of Ramsey ultrafilters extends to any function $f:\omega^n\to\omega$, except for trivial cases.

\begin{prop} Let $\mathfrak u$ be a Ramsey ultrafilter on $\mathscr P(\omega)$, and let $f:\omega^n\to\omega$, $1\leq n<\omega$. Then $\widetilde f(\mathfrak u, \mathfrak u, \ldots, \mathfrak u)=\mathfrak u$ if and only if there are $X\in \mathfrak u$ and $i<n$ such that
$$
f(x_0, x_1, \ldots, x_{n-1})=x_i
$$
for all $x_0< x_1<\ldots<x_{n-1}\in X$.
\end{prop}
\begin{proof} 
%We will use Theorem \ref{Th2} and, in addition, repeat some arguments of its proof. 
Let $\widetilde f(\mathfrak u, \mathfrak u, \ldots, \mathfrak u)=\mathfrak u$, and let $Y\in \mathfrak u$ and $I\subseteq n$ be such that $f$ is selectively upward injective on $Y$ w.r.t.~$I$. \par If $|I|=0$ we have $\widetilde f(\mathfrak u, \mathfrak u, \ldots, \mathfrak u)\in \omega$ by Proposition \ref{Prop1}, a contradiction. \par Let $|I|\geq 1$. Then $$\widetilde f(\mathfrak u, \mathfrak u, \ldots, \mathfrak u)\approx_\mathrm{RK} \underbrace{\mathfrak u\times \mathfrak u\times\ldots\times\mathfrak u}_{\text{$m$ times}}$$
by Theorem \ref{Th2}. If $m\geq 2$, we come to a contradiction again because $\mathfrak u<_{\mathrm{RK}} \mathfrak u\times\mathfrak v$ for all $\mathfrak u, \mathfrak v\in \bm\beta\omega\setminus\omega$ where ${<_{\mathrm{RK}}}={\leq_{\mathrm{RK}}}\setminus {\leq^{-1}_{\mathrm{RK}}}$, see~\cite{Hindman Strauss}. Therefore, $I=\{i\}$ for some $i<n$. As in the proof of Theorem \ref{Th2}, let us construct a function $g:\omega\to\omega$ such that
$f(x_0, x_1, \ldots, x_{n-1})=g(x_i)$ for all $x_0<x_1<\ldots<x_{n-1}\in Y$. We have $\widetilde f(\mathfrak u, \mathfrak u, \ldots, \mathfrak u)=\widetilde g(\mathfrak u)=\mathfrak u$ by Proposition~\ref{Prop1}. Therefore, there is a set $Z\in \mathfrak u$ such that $g(x)=x$ for all $x\in Z$, see~\cite{Hindman Strauss}. So, for all $x_0< x_1< \ldots< x_{n-1}\in Y\cap Z$ we have
$$
f(x_0, x_1, \ldots, x_{n-1})=x_i.
$$
\par In the opposite direction, the proposition immediately follows from Proposition \ref{Prop1}.
\end{proof}

\paragraph*{Discussion.} The shortest and most elegant proof of the Ramsey Theorem uses the ultrafilter technique, see \cite{Graham2015}. Can Theorem \ref{main_var} (or the Canonical Ramsey Theorem) be proved in a similar way? 

\paragraph*{Acknowledgment.} The author thanks Prof. D. I. Saveliev for fruitful discussions.

\end{document}